\documentclass[12pt, reqno]{amsart}
\usepackage{amsmath, amsthm, amscd, amsfonts, amssymb, graphicx, color}
\usepackage[bookmarksnumbered, colorlinks, plainpages]{hyperref}
\hypersetup{colorlinks=true,linkcolor=red, anchorcolor=green, citecolor=cyan, urlcolor=red, filecolor=magenta, pdftoolbar=true}

\usepackage{tensor}
\usepackage{mathrsfs}
\usepackage{cleveref}

\newtheorem{theorem}{Theorem}[section]
\newtheorem{lemma}[theorem]{Lemma}

\theoremstyle{definition}
\newtheorem{definition}[theorem]{Definition}
\newtheorem{example}[theorem]{Example}

\theoremstyle{remark}
\newtheorem{remark}[theorem]{Remark}
\numberwithin{equation}{section}

\usepackage{fullpage}

\begin{document}

\setcounter{page}{1}

\title[MPP on QPMs]{startpoints via weak contractions}

\author[Collins Amburo Agyingi]{Collins Amburo Agyingi$^{1,3}$}

\author[Ya\'e Ulrich Gaba]{Ya\'e Ulrich Gaba$^{1,2,3,*}$}

\address{$^{1}$ Department of Mathematical Sciences, North West University, Private Bag
	X2046, Mmabatho 2735, South Africa.}

\address{$^{2}$ Institut de Math\'ematiques et de Sciences Physiques (IMSP), 01 BP 613 Porto-Novo, B\'enin.}

\address{$^{3}$ African Center for Advanced Studies (ACAS),
	P.O. Box 4477, Yaounde, Cameroon.}

\bigskip

 \email{\textcolor[rgb]{0.00,0.00,0.84}{ collins.agyingi@nwu.ac.za
}}

\email{\textcolor[rgb]{0.00,0.00,0.84}{yaeulrich.gaba@gmail.com
}}



\date{Received: xxxxxx; Accepted: zzzzzz.
\newline \indent $^{*}$Corresponding author}

\begin{abstract}
Startpoints (resp. endpoints) can be defined as ``oriented fixed points". They arise naturally in the study of fixed for multi-valued maps defined on quasi-metric spaces. 
In this article, we give a new result in the startpoint theory for quasi-pseudometric spaces. The result we present is obtained via a generalized weakly contractive set-valued map.
\end{abstract} 

\maketitle

\section{Introduction and preliminaries}

Recently, the notion of a startpoint has been introduced
and studied as a generalization of that of fixed point of set-valued mappings. This notion has
been defined by Gaba\cite{rico} as follows:

\begin{definition}\label{definition1}(Compare \cite{rico})
	Let $(X,d)$ be a $T_0$-quasi-metric space.
Let $F:X\to 2^X$ be a set-valued map. An element $x\in X$ is said to be 
	\begin{enumerate}
		\item[(i)] a startpoint of $F$ if $H(\{x\},Fx)=0$,
		\item[(ii)] an endpoint of $F$ if $H(Fx,\{x\})=0$.
		
	\end{enumerate}
\end{definition}

In Definition \ref{definition1}, $H$ refers to the Hausdorff quasi-pseudometric which we define below. 

The theory of startpoint came to extend the idea of fixed points for multi-valued mappings defined on quasi-pseudometric spaces. A more detailed introduction to the subject can be read in \cite{rico, rico1, ricoo}. As mentioned in Definition \ref{definition1}, the appropriate framework for the theory of startpoint is the quasi-metric setting. For the convenience of the reader, we recall the following well known definitions and facts about quasi-metric spaces as well as some additional definitions related to set-valued maps on these spaces.

\begin{definition}(See \cite{rico})
Let $X$ be a non empty set. A function $d:X \times X \to [0,\infty)$ is called a \textbf{quasi-pseudometric} on $X$ if:
\begin{enumerate}
\item[i)] $d(x,x)=0 \quad \forall \ x \in X$, 
\item[ii)] $d(x,z) \leq d(x,y) + d(y,z) \quad \forall\  x,y,z \in X $. 
\end{enumerate}

Moreover, if

\begin{enumerate}
\item[iii)]  $d(x,y)=0=d(y,x) \Longrightarrow x=y$, then $d$ is said to be a \textbf{$T_0$-quasi-metric}.
\end{enumerate}
 The latter condition is referred to as the $T_0$-condition.
\end{definition}

\begin{remark} \hspace*{0.5cm}  
\begin{itemize}
\item Let $d$ be a quasi-pseudometric on $X$, then the function $d^{-1}$ defined by $d^{-1}(x,y)=d(y,x)$ whenever $x,y \in X$ is also a quasi-pseudometric on $X$, called the \textbf{conjugate} of $d$. 
\item It is easy to verify that the function $d^s$ defined by $d^s:=d\vee d^{-1}$, i.e. $d^s(x,y)=\max \{d(x,y),d(y,x)\}$ defines a metric on $X$ whenever $d$ is a $T_0$-quasi-metric on $X$.
\end{itemize}
\end{remark}

Let $(X,d)$ be a quasi-pseudometric space. For $x \in X$ and $\varepsilon > 0$, $$B_{d}(x,\varepsilon)=\lbrace y \in X: d(x,y) < \varepsilon \rbrace $$ denotes the open $\varepsilon$-ball at $x$. The collection of all such balls yields a base for the topology $\tau (d)$ induced by $d$ on $X$.

\begin{definition} (See \cite{rico}) Let $(X,d)$ be a quasi-pseudometric space. The sequence $(x_n)$ \textbf{$d$-convergences} to $x$, or \textbf{left-convergence} to $x$, and we denote by $x_n \overset{d}{\longrightarrow} x$, if $ d(x_n,x) \longrightarrow 0$.
	\end{definition}

	\begin{definition}
		In a quasi-pseudometric space $(X,d)$, we shall say that a sequence $(x_n)$ \textbf{$d^s$-converges} to $x$ 
		 and we denote it as $x_n \overset{d^{s}}{\longrightarrow} x$ or $x_n \longrightarrow x$ when there is no confusion,
	if
	\[
	  x_n \overset{d}{\longrightarrow} x \ \text{ and }\ x_n \overset{d^{-1}}{\longrightarrow} x.
	\]
	
\end{definition}

\begin{definition}(See \cite{rico})
	A sequence $(x_n)$ in a quasi-pseudometric $(X,d)$ is called
 \textbf{left $K$-Cauchy} if for every $\epsilon >0$, there exists $n_0 \in \mathbb{N}$ such that 
		$$ \forall \  n,k: n_0\leq k \leq n \quad d(x_k,x_n )< \epsilon  ;$$
		
\end{definition}

\begin{definition}(See \cite{rico})
	A quasi-pseudometric space $(X,d)$ is called
	
	\begin{itemize}
		\item \textbf{left-$K$-complete} provided that any left $K$-Cauchy sequence is $d$-convergent,
		\item {\bf Smyth complete} if any left $K$-Cauchy sequence is $d^s$-convergent.
	\end{itemize} 
\end{definition}

We know that every Smyth-complete quasi-metric space is left $K$-complete. It is known that the converse implication does not hold.

\begin{definition}\cite{rico}
 A $T_0$-quasi-metric space $(X,d)$ is called \textbf{bicomplete} provided that the metric $d^s$ on $X$ is complete.
\end{definition}

\begin{definition}\cite{rico}
	Let  $(X,d)$ be a quasi-pseudometric space and $A\subseteq X$. Then $A$ is said to be {\it bounded} provided that there exists a positive real constant $M$ such that $d(x,y)<M$ whenever $x,y\in A$.  
 
\end{definition}

Let $(X,d)$ be a quasi-pseudometric space. We set $\mathscr{P}_0(X):=2^X \setminus \{ \emptyset\}$ where $2^X$ denotes the power set of $X$. For $x\in X$ and $A \in \mathscr{P}_0(X)$, we set:

$$ d(x,A):= \inf\{ d(x,a),a\in A\} , \quad  d(A,x):= \inf\{ d(a,x),a\in A\}.$$

We also define the map $H:\mathscr{P}_0(X) \times \mathscr{P}_0(X) \to [0,\infty]$ by 

$$H(A,B)= \max \left\lbrace \underset{a\in A}{\sup}\ d(a,B), \ \underset{b\in B}{\sup} \ d(A,b)   \right\rbrace \text{ whenever } A,B, \in \mathscr{P}_0(X).$$

Then $H$ is an extended\footnote{This means that $H$ can attain the value $\infty$ as it appears in the definition.} quasi-pseudometric on $\mathscr{P}_0(X)$.

 We shall denote by $CB(X)$ the collection of all nonempty bounded and closed subsets of $(X,d^s)$. We denote by $K(X)$ the family of nonempty compact subsets of $(X, d ^s )$\footnote{They will be called join-compact.}
and by $C(X)$ the family of nonempty closed subsets of $(X,d)$.

We complete this section by recalling the following lemma.

\begin{lemma} (See \cite[Lemma 8]{rico})
	Let $(X,d)$ be a quasi-pseudometric space. For every fixed $x\in X,$ the mapping $y\mapsto d(x,y)$ is $\tau(d)$ upper semicontinuous($\tau(d)$-usc in short) and $\tau(d^{-1})$ lower semicontinuous( $\tau(d^{-1})$-lsc in short). For every fixed $y\in X,$ the mapping $x\mapsto d(x,y)$ is $\tau(d)$-lsc and $\tau(d^{-1})$-usc.
	
\end{lemma}

\section{The result}

In this section, we give a new startpoint theorem for a generalized weakly contractive set-valued map.

We begin with the following intermediate result.

\begin{lemma} \label{lemma1}
	Let $(X,d)$ be $T_0$-quasi-metric space and $A \subset X$. If $A$ is a compact subset of $(X, d^s )$, then it is
	a closed subset of $(X, d)$. That is, $K(X) \subset C(X)$.
\end{lemma}

\begin{proof}
	Let $\{x_n\}$ be a sequence in $A$ such that $d(x, x_n ) \to  0$ for
	some $x \in X$. Since $A$ is a compact subset of $(X, d^s )$, there exist a subsequence $\{x_{n_k} \}$ of $\{x_n\}$ and a point $z \in A$ such that $d^s(z, x_{n_k} ) \to  0$. Thus we have $d(x_{n_k},z ) \to  0$. Using the triangle inequality, we have
	\[  d(x, z) \leq d(x,x_{n_k} )+ d(x_{n_k},z ).  \]
	Letting $k \to \infty$ in above inequality, we get $x = z$ and $x \in A$. Thus $A$ is a
	closed subset of $(X, d)$.
\end{proof}

One of the generalizations of contractions on metric spaces is the concept of weakly contractive
maps which appeared in \cite[Definition 1.1]{alber}. There the authors defined such maps for single valued
maps on Hilbert spaces and proved the existence of fixed points. Rhoades\cite{rhoades}
showed that most results of \cite{alber} still hold in any Banach space. Naturally, this concept can be extended to multi-valued maps on quasi-metrics. In \cite{rico}, we introduced the family of the so-called $(c)$-comparison functions. Here, we give a modified version of these by defining a more refined class of functions, namely that of the $(c)^*$-comparison functions, the purpose being to define weakly contractive multi-valued maps and give a
startpoint theorem for weakly contractive multi-valued maps. 
\begin{definition}
	A function $\gamma : [0,\infty) \to [0,\infty)$ is called a $(c)^*$-comparison function if 
	\begin{enumerate}
		\item[($\gamma_1$)] $\gamma$ is nondecreasing with $\gamma(0)=0$ and $0 < \gamma(t) < t$ for each $t > 0$;
		\item[($\gamma_2$)] for any sequence $\{t_n \}$ of $(0, \infty)$, $\overset{\infty}{\underset{n=1}{\sum}}\gamma(t_n) < \infty$ implies $\overset{\infty}{\underset{n=1}{\sum}}t_n < \infty$. 

	\end{enumerate}
	
\end{definition}

\begin{definition}
	Let $(X,d)$ be $T_0$-quasi-metric space.
	
	\begin{itemize}
	
	\item[1.] A set-valued map $F : X \to 2^X$ is called weakly contractive if, for each $x\in X$ and a given $(c)^*$-comparison function $\gamma$, there exists $y\in Fx$

	\begin{equation}\label{wcm}
		H (y, F y)\footnote{This is a short form for $H (\{y\}, F y)$ } \leq d(x, y) -\gamma(d
		(x, y)).
	\end{equation}

\item[2.] A single valued map $f : X \to X$ is called weakly contractive if, for each $x, y in X$ 
\begin{equation}\label{wcs}
d (f x, f y) \leq d(x, y) -\gamma(d(x, y)).
\end{equation}

\end{itemize}
\end{definition}

Now, we state our result.

\vspace*{0.2cm}

\begin{theorem}\label{theorem1}
Let $(X,d)$ be a left $K$-complete quasi-pseudometric space, $F:X\to CB(X)$ be a weakly contractive set-valued map, then $F$ has a  startpoint in $X$.
\end{theorem}

\begin{proof}
	Let $x_0 \in X$, by \eqref{wcm} there exists $x_1 \in Fx_0$, 
	and
	 there exists\footnote{In fact for any $x_2 \in Fx_1$.} a $x_2 \in Fx_1$ such
	that

	\[  d(x_1 , x_2 )\leq H(x_1,Fx_1) \leq d(x_0 , x_1 )- \gamma(d(x_0 , x_1 )). \]
	
	Again by \eqref{wcm}, there exists an $x_3 \in Fx_2$ such that
	\[  d(x_2 , x_3 )\leq H(x_2,Fx_2) \leq d(x_1 , x_2 )- \gamma(d(x_1 , x_2 )\leq d(x_1,x_2)\leq H(x_1,Fx_1) .
	 \]
	
Continuing this process, we can find a sequence $\{x_n \} \subset X$ such that for $n = 0, 1, 2, \cdots$

\[  x_{n+1} \in Fx_n \]

and

\[d(x_{n+1} , x_{n+2} )\leq H(x_{n+1},Fx_{n+1
}) \leq d(x_n , x_{n+1} ) - \gamma(d(x_n , x_{n+1} ))\leq H(
x_{n},Fx_n
). \]

Thus the sequence $\{d(x_n , x_{n+1} )\}$ is nonincreasing and so $\lim\limits_{n\to \infty} d(x_n , x_{n+1} ) =
l$ for some $l \geq 0$.	We show that $l = 0$. Suppose $l > 0$. Then we have
	
	\[   d(x_n , x_{n+1} ) \leq d(x_{n-1} , x_n ) -\gamma(d(x_{n-1} , x_n )) \leq d(x_{n-1} , x_n ) - \gamma(l),\]
	
and so	
	
	\[  d(x_{n+N} , x_{n+N +1} ) \leq d(x_{n-1} , x_n ) - N \gamma(l),\]
	
	which is a contradiction for $N$ large enough. Thus we have
	
	\[ \lim\limits_{n\to \infty} d(x_n , x_{n+1} ) =0.\]

For $m \in \mathbb{N}$ with $m \geq 3$, we have

\begin{align*}
	d(x_{m-1} , x_m )& \leq
	 	d(x_{m-2} , x_{m-1} ) - \gamma (d(x_{m-2} , x_{m-1} ) ) \cdots\\
	&\leq  d(x_1 , x_2 ) - \gamma(d(x_1 , x_2 )) -\cdots · · · - \gamma (d(x_{m-2} , x_{m-1} ) ).
\end{align*}

Hence we get

\[ \sum_{k=1}^{m-2}  \gamma(d(x_k , x_{k+1} )) \leq d(x_1 , x_2 ) - 	d(x_{m-1} , x_m ).\]

Letting $m \to \infty$ in above inequality, we obtain

\[ \sum_{k=1}^{\infty}  \gamma(d(x_k , x_{k+1} )) \leq d(x_1 , x_2 )< \infty,\]

which implies that

\[ \sum_{k=1}^{\infty}  d(x_k , x_{k+1} ) < \infty \text{ by $(\gamma_2)$
 } .\]

We conclude that $\{x_n\}$ is a left $K$-Cauchy sequence. According to the left $K$-completeness of $(X,d),$ there exists $x^* \in X$ such that $x_n \overset{d}{\longrightarrow} x^*$.

Given the function $f(x):= H(x,Fx)$, observe that the sequence $(fx_n)=(H(x_{n},Fx_{n}))$ is decreasing and converges to $0$. Since $f$ is $\tau(d)$-lower semicontinuous (as supremum of $\tau(d)$-lower semicontinuous functions), we have 

\[
0\leq f(x^*) \leq \underset{n\to \infty}{\liminf} f(x_n)=0.
\]
Hence $f(x^*)=0$, i.e. $H(\{x^*\},Fx^*)=0$.

\noindent
This completes the proof.

\end{proof}

\begin{remark}
	The reader can convince him(her)self that if we replace the condition \eqref{wcm} by the dual condition 
	\begin{equation}\label{wcmprime}
	H (Fy, y)\leq d(y, x) -\gamma(d
	(y,x)),
	\end{equation}
	then the conclusion of Theorem \ref{theorem1} would be that the multi-valued function $F$ possesses an endpoint.
	Moreover for the multi-valued function $F$ to admit a fixed point, it is enough that
	 \begin{equation}\label{wcmprimeprime}
	 H^s (Fy, y)\leq \min \{ d(x,y) -\gamma(d
	 (x,y
	 )), d(y, x) -\gamma(d
	 (y,x)) \}.	 
	 \end{equation}
\end{remark}

We conclude this paper with the following illustrative example:

\begin{example}
Let $$X = \left\lbrace \frac{1}{2^n} : n = 0, 1, 2, \cdots \right\rbrace \cup \{0\}$$ and let

$$
d(x, y) =
\begin{cases}
y - x,\ \ \ \ \ \ \ \ \ \ \ 
\text{if } y\geq x,\\
2(x - y), \ \ \ \ \ \ \ 
\text{if } x > y. 
\end{cases}
$$
Then $(X, d)$ is a left $K$-complete $T_0$-quasi-metric space. Let $ \gamma(t) = \frac{t}{2}$ for all $t \geq 0$ and let $F : X \to CB(X)$ be a set-valued map
defined as

$$
Fx =
\begin{cases}
\left\lbrace \frac{1}{2^{n+1}},0 \right\rbrace  \ \ \ 
\text{if } x=\frac{1}{2^n} : n = 0, 1, 2, \cdots ,\\
\{0\}, \ \ \ \ \ \ \ 
\text{if } x=0. 
\end{cases}
$$
We now show that $F$ satisfies condition \eqref{wcm}.

\begin{itemize}
	\item[Case 1.] $x=0$, there exists $y= 0 \in F x=F0=\{0\}$ such that
	\[ 0=H (y, F y)=H(0,F0)\leq d(0, 0) -\gamma(d
	(0, 0))=0.  \]
	
 \item[Case 2.] $x=\frac{1}{2^n}$, there exists $y= 0 \in F x=\left\lbrace \frac{1}{2^n},0 \right\rbrace$ such that
 \[ 0=H (y, F y)=H(0,F0)\leq d\left(\frac{1}{2^n}, 0\right) -\gamma\left(d\left(\frac{1}{2^n}, 0\right)\right).  \]
\end{itemize}
The map $F$ satisfies the assumptions of Theorem \ref{theorem1}, so it has a startpoint, which in this case is $0$.

\end{example}

\section*{Conflict of interest.}

The authors declare that there is no conflict
of interests regarding the publication of this article.

\bibliographystyle{amsplain}

\end{document}